\documentclass{amsart}
\usepackage{amsmath,amssymb,amsfonts,verbatim}

\theoremstyle{plain}
\newtheorem{theorem}{Theorem}
\newtheorem{lemma}[theorem]{Lemma}
\newtheorem{proposition}[theorem]{Proposition}
\newtheorem{corollary}[theorem]{Corollary}
\newtheorem{definition}[theorem]{Definition}

\makeatletter
\@addtoreset{equation}{section}
\makeatother

\renewcommand\a{\alpha}
\renewcommand\b{\beta}
\newcommand\Dp{\mathcal{P}}
\newcommand\dd{\delta}

\renewcommand\ge{\geqslant}
\newcommand\h{\mathcal H}
\renewcommand\H{\widetilde{\mathcal H}}
\renewcommand\l{\lambda}
\renewcommand\L{\Lambda}
\renewcommand\le{\leqslant}
\renewcommand\o{\omega}
\newcommand\Rl{\mathcal{R}}
\newcommand\smm{s_{i_{1}}\cdots s_{i_{m}}}

\renewcommand\v{\varphi}

\def\CARRE#1{\hbox{\vrule width \thickness
   \vbox to \carresize{\hrule height \thickness\vss
      \hbox to \carresize{\hss#1\hss}
   \vss\hrule height\thickness}
\unskip\vrule width \thickness} \kern-\thickness}
\def\vsquare#1{\vbox{\CARRE{$#1$}}\kern-\thickness}
\def\blk{\omit\hskip\carresize}

\def\young#1{%
  \newdimen\carresize \carresize=16pt%
  \newdimen\thickness \thickness=0.5pt%
  \vcenter{%
    \vbox{\smallskip\offinterlineskip%
      \halign{&\vsquare{##}\cr #1}}}}

\begin{document}

\title[Representations of affine Hecke algebras]
{Cyclic generators for irreducible representations 
of affine Hecke algebras}

\author[Guizzi]{Valentina Guizzi}
\address{
Dipartimento di Economia,
Universit\`a Roma Tre,
Via Silvio D'Amico 77, 00145 Roma,
Italy;}
\email{guizzi@uniroma3.it}

\author[Nazarov]{Maxim Nazarov}
\address{
Department of Mathematics,
University of York,
York YO10 5DD,
England;}
\email{mln1@york.ac.uk}

\author[Papi]{Paolo Papi}
\address{
Dipartimento di Matematica,
Sapienza Universit\`a di Roma,
Piazzale Aldo Moro 2, 00185 Roma, Italy;}
\email{papi@mat.uniroma1.it}
\keywords{Affine Hecke algebras, Young diagrams, fusion procedure.}


\begin{abstract}
We give a detailed account of a combinatorial
construction, due to Cherednik,  of cyclic generators
for irreducible modules of the affine
Hecke algebra of the general linear group with generic parameter $q$.
\end{abstract}

\maketitle


\thispagestyle{empty} 

\section*{Introduction}

Let $F$ be a non-Archimedean local field.
Denote by $\mathcal{O}$ the ring of integers in $F$.
Let $P$ be the maximal ideal of the ring $\mathcal{O}$.
Take the general linear group $GL_n(F)$.
The group $GL_n(\mathcal{O})$ is a maximal compact subgroup
of $GL_n(F)$. Now consider
the \textit{Iwahori subgroup} $J\subset GL_n(F)$.
It consists of all matrices from $GL_n(\mathcal{O})$
whose entries below the main diagonal belong to 
$P$. By definition, the
\textit{affine Hecke algebra} $\H_n$ 
consists of $J\,$-biinvariant compactly supported functions on  
the group $GL_n(F)$ with complex values. The multiplication
on $\H_n$ is the convolution~of~functions.
 
The complex associative algebra $\H_n$ admits a remarkable
presentation, due to Bernstein. It is generated by the
elements $T_1,\ldots,T_{n-1}$ and invertible elements
$X_1,\ldots,X_n$ subject to relations
\eqref{H1}-\eqref{X3}. Here $q$ is the cardinality of the residue 
field $\mathcal{O}/P$.
The subalgebra of $\H_n$ generated solely
by the elements $T_1,\ldots,T_{n-1}$ can be then identified as
the subalgebra of functions supported on $GL_n(\mathcal{O})$.
It is sometimes called the \textit{finite Hecke algebra};
in this article we denote it by $\h_n$.

For any representation $W$ of the group $GL_n(F)$
consider the subspace $W^J$ in $W$ consisting of
the vectors fixed by the action of the Iwahori subgroup $J$.
The algebra $\H_n$ acts on the subspace $W^J$ by definition.
The correspondence $W\mapsto W^J$ is an equivalence
between the category of representations of $GL_n(F)$
generated by their subspaces of $J$-fixed vectors, and the category
of all $\H_n\,$-modules.
Furthermore, all irreducible $\H_n$-modules
are finite-dimensional,
see for instance \cite{H} and \cite{R}.

Results of Bernstein and Zelevinsky \cite{BZ,Z}
provide a classification of irreducible $\H_n$-modules.
As is explained for instance in \cite{LNT}, it suffices
to classify only those irreducible $\H_n$-modules
where all the eigenvalues of $X_1,\ldots,X_n$ belong to 
$q^{\,\mathbb Z}$. The latter
$\H_n$-modules are labeled by 
the combinatorial objects called \text{multisegments}.

A \textit{multisegment} is a formal
finite unordered sum of intervals in $\mathbb Z$,
\begin{equation}
\label{museg}
M=\sum_{i\le j}m_{ij}[i,j]
\end{equation}
where the coefficients $m_{ij}$ are non-negative integers.
Zelevinsky's construction is as follows. To a segment
$[i,j]$ one associates the 1-dimensional $\H_{j-i+1}$-module 
$V_{[i,j]}$ where the generator $T_k$ 
acts as $q$ whereas $X_l$ acts as $q^{\,i+l-1}$. 
Given a multisegment $M$,
fix any order on it and consider the tensor product
\begin{equation}
\label{wm}
\mathop{\otimes}\limits_{i\le j}\,
V_{[i,j]}^{\,\otimes m_{ij}}\,
\end{equation}
This is a module over the tensor product of the algebras 
$\H_{j-i+1}^{\,\otimes m_{ij}}$
which can be naturally identified with a subalgebra in $\H_n$.
Now induce \eqref{wm} to $\H_n$. 
For a certain ordering of $M$,
the induced module has a unique non-zero irreducible
submodule. The irreducible $\H_n$-modules
obtained in this way are non-equivalent
for different multisegments $M$ and
form a complete set with
all eigenvalues of $X_1,\ldots,X_n$~in~$q^{\,\mathbb Z}$.

Using the above mentioned ordering of $M$, we 
describe the unique non-zero
irreducible submodule in the $\H_n$-module 
induced from \eqref{wm}.
Our description is more explicit than that given 
by Rogawski \cite{R}, and follows the works of Cherednik \cite{C1,C2}. 
We employ combinatorial objects which
are in a bijection with Zelevinsky multisegments, and which
we call Cherednik diagrams. They are
certain subsets of $\mathbb Z^2$ similar to Young diagrams,
see Definition 1.1.
Moreover, both usual and skew Young diagrams are particular cases 
of Cherednik diagrams. 

Let $\mathcal A_n$ be the subalgebra of $\H_n$ 
generated by $X_1^{\pm 1},\ldots,X_n^{\pm 1}$. This
is a maximal commutative subalgebra of $\H_n$.
For each Cherednik diagram $\l$ we
produce a pair $(E,\chi)$ where $E\in\mathcal H_n$ and  
$\chi:\mathcal A_n\to\mathbb C$ is a character of the algebra 
$\mathcal A_n$, such that: 
\begin{enumerate}
\item[(i)] 
$E$ is  an eigenvector for $\mathcal A_n$ inside 
$\mathrm{Ind\,}^{\H_n}_{\mathcal A_n} \chi=\h_n$;
\item[(ii)]
the space  $\h_n\cdot E$ 
which is a $\H_n$-module by (i),  is  irreducible.
\end{enumerate}

\smallskip\noindent
Our $E$ will be the element $E_\l$ defined by \eqref{E},
and $\chi$ will be $w_0\cdot\chi_\l$
where the character $\chi_\l$ is defined by \eqref{chi}. 
Here $w_0$ is the longest element of $\mathcal S_n$,
and we use the natural action of the group $\mathcal S_n$
on the characters of $\mathcal A_n$.
The elements $E_\l$ where
introduced by Cherednik in \cite {C1} for the degenerate affine
Hecke algebra and then in \cite {C2} for $\H_n$. 
But proofs are not given in \cite{C1,C2}
and the purpose of our paper is to provide them.
The notion of a Cherednik diagram is also taken from \cite{C1,C2}.

From now on we will regard $q$ as a formal parameter.
Thus $\H_n$ will be defined as an algebra 
over the field $\mathbb C(q)$ with generators 
$T_1,\ldots,T_{n-1} ,X_1^{\pm 1},\ldots,X_n^{\pm 1}$ and relations
\eqref{H1}-\eqref{X3}. The subalgebras
$\h_n$ and $\mathcal A_n$ of $\H_n$ then
become $\mathbb C(q)$-algebras too.
It is known that when the
parameter $q\/$ specializes to a non-zero complex number
of infinite multiplicative order, the
parametrization of the irreducible
$\H_n$-modules is the same for any such 
specialization. Our construction of
the element $E_\l\in\h_n$ also
allows any such specialization of 
$q\,$. Moreover, the corresponding specialization of 
the $\H_n$-module $V_\l=\h_n\cdot E_\l$ 
remains irreducible; cf.\ \cite{L3,X1,X2}.

Our construction of the element $E_\l$ is based on the 
{\it fusion procedure} due to Cherednik \cite {C1,C2}. 
Up to normalization, here $E_\l$
is obtained as a limit of
certain $\mathcal H_n$-valued function
$\v_0(x_1,\ldots,x_n)$ of $n$ variables from $\mathbb{C}(q)$. 
This function is a product of elementary
factors \eqref{basic} corresponding to simple transpositions
in a reduced decomposition of the element $w_0\in\mathcal{S}_n$. 
The factors satisfy the Yang-Baxter relations, 
so that the function $\v_0(x_1,\ldots,x_n)$
does not depend on the choice of
the reduced decomposition of $w_0$; see Lemma \ref{L2}.

Fusion procedure is a method that was initially used to reproduce 
the Young symmetrizers in the group ring of $\mathcal{S}_n$.
Nazarov used it in his works on projective 
representations of the symmetric group \cite{N} and on their
$q$-analogues \cite{JN}. In the case when $\l$ is
a usual Young diagram, a detailed construction
of the element $E_\l\in\mathcal{H}_n$ by this method
has been given in \cite{Nejc}. The results of \cite{Nejc}
are easy to generalize to skew Young diagrams.
But here we have to extend the method to those
Cherednik diagrams $\l$, which are neither usual nor skew Young diagrams.
In the corresponding $\H_n$-modules, 
the action of the subalgebra $\mathcal A_n$ is not semisimple.
Indeed, ours seems to be the first instance of a combinatorial 
treatment for these modules. By contrast,
the irreducible $\H_n$-modules with a 
semisimple action of $\mathcal A_n$ are well understood;
they correspond to skew Young diagrams $\l$.
A thoroughful treatment of them
can be found in the work of Ram \cite{Ram}.
Moreover, by using the fusion procedure,
in each of these modules
one can construct a basis of eigenvectors of $\mathcal A_n$,
not only a $\mathcal{H}_n$-cyclic vector.
This result was also stated by Cherednik \cite{C1,C2}.

The emphasis in our paper is on the combinatorial aspects
of fusion procedure; Sections 2 and 3 are devoted to this. 
The key technical difference comparing to the case
of a skew Young diagram $\l$ 
is that to remove the singularity,
the function $\v_0(x_1,\ldots,x_n)$
has to be multiplied by a correction factor 
$\dd_\l(x_1,\ldots,x_n)$; the limit
is taken afterwards. Further, for any skew Young diagram $\l$
one has the equality
\begin{equation}
\label{lead0}
E_\l=T_{w_0}+\sum_{\ell(w)<\ell(w_0)}a_w\,T_w
\end{equation}
for some coefficients $a_w\in\mathbb C(q)$; here $T_w$ are the standard
basis elements of $\mathcal{H}_n$ and $\ell(w)$
is the length function on $\mathcal{S}_n$.
But for an arbitrary Cherednik diagram $\l$ we 
have the equality \eqref{desired}
where the element $w_\l\in \mathcal{S}_n$ may differ from the longest element
$w_0$. We would  also like to emphasize 
that our construction of the element $E_\l\in\mathcal{H}_n$ 
is completely explicit, see Corollaries \ref{limiting}
and \ref{shortening}.

In Section 1 of our paper we fix the notation, state the main theorems and 
prove the irreducibility of the $\H_n$-module  
$V_\l$. 
The irreducibility is proved by a rather indirect approach.  
Using the $q\,$-analogue of Drinfeld functor \cite{D}
due to Cherednik \cite{Ch}, we reduce the argument to 
irreducibility of certain finite-dimensional modules of 
quantum affine algebras, which has been 
proved by Akasaka and Kashiwara \cite{AK}.
The results of \cite{AK} also imply that the $\H_n$-modules
$V_\l$ for different Cherednik diagrams $\l$
are pairwise non-equivalent. Note that the irreducibility
and pairwise non-equivalence of the $\H_n$-modules
$V_\l$ can also be proved using the methods~of~\cite{C3},
see for instance \cite{Vaz}.


\section*{1.\ Fusion procedure}
\setcounter{section}{1}
\setcounter{subsection}{0}
\setcounter{equation}{0}
\setcounter{theorem}{0}


\subsection{Hecke algebras}
Let $\h_n$ and $\H_n$ denote respectively the finite and affine Hecke algebras
of $GL_n$. We define $\h_n$ as the associative algebra
over the field $\mathbb C(q)$ with generators
$T_1,\ldots,T_{n-1}$ and relations
\begin{align}
\label{H1}
&T_kT_{k+1}T_k=T_{k+1}T_kT_{k+1},
&&
1\le k\le n-2,
\\
\label{H2}
&T_kT_l=T_lT_k,
&&1<|k-l|,
\\
\label{H3}
&(T_k-q)(T_k+1)=0,
&& 1\le k\le n-1.
\end{align}
Here $q$ is a formal parameter.
It is well known \cite{GU} that $\h_n$ is a $q$-deformation of the group
algebra of the symmetric group ${\mathcal S}_n$, the generator $T_k$
corresponding to the simple transposition $s_k=(k,k+1)$.
Let $w=s_{k_1}\dots s_{k_m}$ be a reduced decomposition of
an element $w\in {\mathcal S}_n$. By \eqref{H1},\eqref{H2}
the element  $T_w=T_{k_1}\dots T_{k_m}$ does not depend on the choice of
the reduced decomposition of $w$.

We define $\H_n$ as the associative algebra
over $\mathbb C(q)$ generated by $T_1,\ldots,T_{n-1}$
subject to the relations above, and by the invertible
elements $X_1,\ldots,X_n$ such that
\begin{align}
\label{X1}
&X_kX_l=X_lX_k, && 1\le k,l\le n,
\\
\label{X2}
&X_lT_k=T_kX_l,&& l\not= k,k+1,
\\
\label{X3}
&T_kX_kT_k=q\,X_{k+1}, && 1\le k\le n-1 .
\end{align}

Let ${\mathcal A}_n$ be the subalgebra of $\H_n$ generated by
the elements $X_1,\ldots,X_n$ and by their inverses.
Then $\{T_w\mid w\in {\mathcal S}_n\}$ is a
basis of  $\H_n$ both as a right and as a left ${\mathcal A}_n$-module.
Moreover,  ${\mathcal A}_n$ is a maximal commutative subalgebra of $\H_n$.
The centre of $\H_n$ consists of those elements of ${\mathcal A}_n$
which are invariant under the action of the group
${\mathcal S}_n$ by permutations of $X_1,\ldots,X_n$.

An easy way to produce an $\H_n$-module is to induce from an algebra
character $\chi:{\mathcal A}_n\to \mathbb C(q)$.
Then the one-dimensional space $\mathbb C(q)$
can be regarded as a
${\mathcal A}_n$-module via $\chi$ and we can form the induced module
$I_\chi$. As a vector space $I_\chi$ can be identified with $\h_n$
where the algebra $\h_n$ acts via left multiplication.
Hence all $I_\chi$ are isomorphic to each other as $\h_n$-modules.
In $I_\chi$ we also have
$X_k\cdot 1=\chi(X_k)$
for 
$i=1,\ldots,n$.
We shall realize our irreducible $\H_n$-modules as
cyclic submodules of certain induced modules $I_\chi$.
When the
parameter $q\/$ specializes to any non-zero complex number
of infinite multiplicative order, our submodules 
will remain irreducible. Moreover, then they will make a complete set
of irreducible pairwise non-equivalent modules over
any such specialization of 
$\H_n$. To produce appropriate characters $\chi$
we need combinatorial tools, which we introduce next.


\subsection{Cherednik diagrams}
\label{1.2}
Zelevinsky \cite{Z} employed the
multisegments \eqref{museg} to
parametrize the irreducible representations of
the group $GL_n(F)$ generated by their subspaces
of the vectors fixed by the action of the Iwahori subgroup $J$. 
It is useful to
introduce an ordering on a multisegment. This leads to considering
the following combinatorial object, which we call a {\it Cherednik
diagram}.

Let $(a_1,\ldots,a_r)$ and $(b_1,\ldots,b_r)$
be sequences of integers with the same number of terms,
such that $a_i\le b_i$ for each index $i$. Consider the set
\begin{equation}
\label{diag}
\{(i,j)\in\mathbb Z^2\mid 1\le i\le r,\,a_i\le j\le b_i\}.
\end{equation}
The total number of elements in the set \eqref{diag} is
called its \textit{degree} and denoted by~$n$,
$$
\sum_{i=1}^r(b_i-a_i+1)=n.
$$
Then we can graphically represent the set \eqref{diag}
by a diagram with $n$ boxes arranged in $r$ rows,
consisting respectively of $b_1-a_1+1,\dots,b_r-a_r+1$ boxes.
Using the matrix-style coordinates on the plane $\mathbb R^2$,
the element $(i,j)$ of \eqref{diag} is represented by a unit box on
$\mathbb R^2$ with the lower right corner placed at the point
$(i,j)$. Let $\mathcal C_n$ denote
the collection of sets \eqref{diag} of degree $n$,
represented by their diagrams.

\begin{definition}
\label{CD}
The set \eqref{diag} 
is a {Cherednik diagram} 
if for each $i=1,\ldots,r-1$ either\/ $b_{i+1}\le b_i$, or
$b_{i+1}=b_i+1$ and $a_{i+1}\le a_i+1\,$.
\end{definition}

Let $\mathcal C_n $ denote the collection of  Cherednik diagrams
of degree $n$.  It is clear that $\mathcal C_n$ contains all ordinary and
skew Young diagrams of degree $n$. But the set $\mathcal C_n$
also contains other diagrams. For instance, the set $\mathcal C_7$ contains
the three diagrams

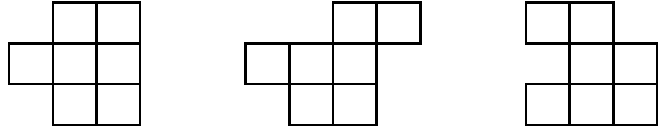
\begin{equation*}
\text{$\young{\blk&&\cr&&\cr\blk&&\cr}$}
\qquad\qquad
\text{$\young{\blk&\blk&&\cr&&&\blk\cr\blk&&&\blk\cr}$}
\qquad\qquad
\text{$\young{&&\blk&\blk\cr\blk&&&\blk\cr&&&\blk\cr}$}
\end{equation*}

\bigskip\noindent
Here and in what follows we identify diagrams
with their graphical representations.
The next two diagrams do \textit{not\/}
belong to $\mathcal C_5$ and $\mathcal C_3$ respectively:

$$
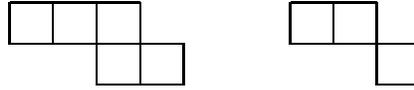
\text{$\young{&&&\blk\cr\blk&\blk&&\cr}$}
\qquad\qquad
\text{$\young{&&\blk\cr\blk&\blk&\cr}$}
$$

\bigskip\noindent
We do allow ``disconnected" diagrams.
For instance, $\mathcal C_4$ contains the diagram

$$
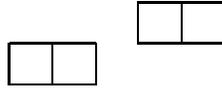
\young{\blk&\blk&\blk&&\cr&\cr}
$$

\bigskip\noindent
For the set \eqref{diag},
we will use the partition-like notation
$\l\!=\!(\l_1,\ldots,\l_r)$ where
$\l_i=[a_i,b_i]$ and
$1\le i\le r$. Moreover, we will then write
$|\l|=n$ if
$\l$ has degree $n$.

Let us denote by
$\mathcal M_n$ the set of multisegments \eqref{museg} such that
$$
\sum_{i\le j} m_{ij}(j-i+1)=n.
$$
There is a natural bijection $\mathcal M_n\to\mathcal C_n$.
Indeed, given a multisegment $M\in\mathcal M_n$,
consider the multiset
$$
\{\,\ldots,\underbrace{[i,j],\dots,[i,j]}_{m_{ij}},\ldots\,\}
$$
and endow it with the reverse lexicographical order,
hence obtaining an ordered collection of segments
$(\,[i_1,j_1],\ldots,[i_r,j_r]\,)$. Form a diagram
$\l$ by setting
$$
\l_k=[a_k,b_k]=[i_k+k,j_k+k]
$$
for each $k=1,\ldots,r$. We claim that $\l\in\mathcal C_n$.
Suppose that $b_k<b_{k+1}$. This means that $j_k\le j_{k+1}$.
Either $[i_k,j_k]$ precedes $[i_{k+1}, j_{k+1}]$ in the reverse
lexicographical order, or these two segments are equal to each other.
 Hence $j_k\ge j_{k+1}$, so that $j_k=j_{k+1}$
and $b_k+1=b_{k+1}$. In that case $i_k\ge i_{k+1}$, so that
$a_k+1\ge a_{k+1}$. Thus the conditions of Definition \ref{CD}
are satisfied. The map $M\mapsto\l$ is clearly invertible.

For a diagram $\l\in\mathcal C_n$, we denote by the corresponding upper case
letter $\L$ the {\it row filling} of $\l$.
This is the tableau obtained by filling  $\l$ with the numbers
$1,\ldots,n$ first from left to right and then from top to bottom.
For instance,

\begin{equation}
\label{diaggg}
\l\,=\
\young{&\cr\blk&&\cr\blk&&\cr}
\qquad\qquad
\L\,=\
\young{1&2\cr\blk&3&4\cr\blk&5&6\cr}
\end{equation}

\medskip
\begin{definition}
Two rows $\l_i=[a_i,b_i]$ and $\l_j=[a_j,b_j]$ of a diagram
$\l\in\mathcal C_n$ are said to be {\it parallel} if\/
$a_i-i=a_j-j$ and $b_i-i=b_j-j$.
\end{definition}

\noindent
For instance, the first two rows in the diagram $\l$ in the display
\eqref{diaggg} are parallel. Note that under the map $M\mapsto\l$
defined above, parallel rows of the Cherednik diagram $\l$
correspond to identical segments of the Zelevinsky multisegment $M$.
The following lemma follows directly from the conditions of
Definition~\ref{CD}.

\begin{lemma}
\label{comblemma}
Let $\l\in\mathcal C_n$. Suppose the rows $\l_i$ and
$\l_j$ with $i<j$ end on the same diagonal.
Then the same happens for all rows $\l_k$ with $i\le k\le j$.
If moreover $\l_i$ is parallel to $\l_j$, then all rows
$\l_k$ with $i\le k\le j$ are parallel to each other.
\end{lemma}


\subsection{The function $\v_\l$}
\label{thefun}
For any $x\in\mathbb C(q)$ put
$\langle x\rangle=(1-q)/(1-x)$.
For each $k=1,\ldots,n-1$
introduce
the rational function of the variables $x_1,\ldots,x_n\in\mathbb C(q)$
\begin{equation}
\label{basic}
\v_k(x_1,\ldots,x_n)=
T_k+\langle x_{k+1}/{x_k}\rangle
\end{equation}
with values in the algebra $\h_n$. For each permutation
$w\in {\mathcal S}_n$ and
any rational function $\psi$ of $x_1,\ldots,x_n$ with values
in $\h_n$ we will write
$$
\phantom{}^w\psi\,(x_1,\ldots,x_n)=
\psi\,(x_{w(1)},\ldots,x_{w(n)}).
$$
Using any reduced decomposition $w=s_{k_1}\ldots s_{k_m}$ in ${\mathcal S}_n$
define the rational function
\begin{equation}
\label{phiw}
\v_w=\v_{k_1}
(\phantom{}^{s_{k_1}}\v_{k_2})
(\phantom{}^{s_{k_1}s_{k_2}}\v_{k_3})
\dots
(\phantom{}^{s_{k_1}\dots s_{k_{m-1}}}\v_{k_m}).
\end{equation}
For instance, if $w=s_1s_2s_3$ we have
$$
\v_w=\left(T_1+\langle{x_2}/{x_1}\rangle\right)
\left(T_2+\langle{x_3}/{x_1}\rangle\right)
\left(T_3+\langle{x_4}/{x_1}\rangle\right).
$$
The function $\v_w$ does not depend on the choice of a reduced
decomposition of $w$. 
The independence follows from the next lemma, proved by a direct computation.

\begin{lemma}\label{L2}
We have equality of rational functions in
$x,y,z,x,'y',z'\in\mathbb C(q)$,
$$
(T_k+\langle x\rangle)\,
(T_{k+1}+\langle y\rangle)\,
(T_k+\langle z\rangle)\,=\,
(T_{k+1}+\langle z'\rangle)\,
(T_k+\langle y'\rangle)\,
(T_{k+1}+\langle x'\rangle)
$$
if and only if $x=x'$, $z=z'$ and $y=y'=xz$.
\end{lemma}

Now let $\l\in \mathcal C_n$ and $\L$ be the row filling of $\l$.
For every box $(i,j)$ of $\l$ the difference $j-i$ is called
the {\it content\/} of this box. For $k=1,\ldots,n$ denote by $c_k$
the content of the box which is filled with the number $k$ in $\L$.
Denote by $\delta_\l(x_1,\ldots,x_n)$ the product
of the differences $1-x_l/x_k$ taken over all
pairs $(k,l)$ such that $k<l$ while in $\L$ the numbers $k,l$
occur in the leftmost boxes of two parallel rows of $\l$.
We assume that $\delta_\l(x_1,\ldots,x_n)=1$ 
if $\l$ does not have distinct  parallel rows.

Let $\mathcal F_\l$ be the affine subspace in $\mathbb C(q)^{\times n}$
consisting of all points $(x_1,\ldots,x_n)$ such that
$x_k\,q^{c_l}=q^{c_k}\,x_l$ whenever 
$k$ and $l$ are
in the same row of $\L$. Consider the rational function \eqref{phiw}
corresponding to the element $w_0\in {\mathcal S}_n$ of maximal length,
\begin{equation}\label{basicf}
\v_0(x_1,\ldots,x_n)=\v_{w_0}(x_1,\ldots,x_n)\,.
\end{equation}
The following
theorem strengthens a classical result of Cherednik \cite[Theorem 1]{C2}.

\begin{theorem}
\label{T3}
For any $\l\in\mathcal C_n$
the restriction of the rational function
$\delta_\l\,\v_0$ to the
subspace $\mathcal F_\l$ is regular and non-zero at the point
\begin{equation}
\label{point}
(x_1,\ldots,x_n)=(q^{c_1},\ldots,q^{c_n})\,.
\end{equation}
\end{theorem}

So we can take the value of the restriction of $\delta_\l\,\v_0$
to $\mathcal F_\l$ at the point \eqref{point},
\begin{equation}
\label{E}
E_\l=(\delta_\l\,\v_0)|_{\,\mathcal F_\l}
(q^{c_1},\ldots,q^{c_n})\,.
\end{equation}
Theorem \ref{T3} is proved in Sections 2 and 3.
We first show that the restriction to $\mathcal F_\l$ of the function
$\dd_\l \v_0$ is regular at the point \eqref{point}, see
Proposition~\ref{regular}. Then we show that
the corresponding value is non-zero, see Proposition~\ref{enonzero}.
An important role of the non-zero element $E_\l\in\h_n$
is explained in the following subsection.


\subsection{Intertwining operators}

For each index $k=1,2,\ldots,n-1$ denote
$$
\Phi_{k}=T_{k}+(1-q)/(1-X_{k}X_{k+1}^{-1}).
$$
Then $\Phi_1,\ldots,\Phi_{n-1}$ lie in the localization of the
algebra $\H_n$ relative to the set of denominators
$$
\{1-X_{k}X_{l}^{-1}\mid 1\le k,l\le n,\,k\neq l\,\}.
$$
The following relations imply, in particular,
that the Ore conditions are satisfied:
\begin{align}
\label{PX1}
&\Phi_kX_k=X_{k+1}\Phi_k,
&&
\\
\label{PX2}
&\Phi_kX_{k+1}=X_k\Phi_k,
&&
\\
\label{PX3}
&\Phi_kX_l=X_l\Phi_k,
&&
l\neq k,k+1;
\end{align}
see for instance \cite{L}.
By \eqref{H1} to \eqref{X3} we also have relations in the ring of fractions,
\begin{align}
\label{PP1}
&\Phi_k\Phi_{k+1}\Phi_k=\Phi_{k+1}\Phi_k\Phi_{k+1},
&&
1\le k\le n-2,
\\
\label{PP2}
&\Phi_k\Phi_l=\Phi_l\Phi_k,
&&1<|k-l|,
\\
\nonumber
&
\Phi_k^2=
(q-X_kX_{k+1}^{-1})(1-q\,X_kX_{k+1}^{-1})/(1-X_kX_{k+1}^{-1})^2,
&& 1\le k\le n-1.
\end{align}

By using any reduced recomposition $w=s_{k_1}\dots s_{k_m}$ 
in ${\mathcal S}_n$,
we can define an element $\Phi_w=\Phi_{k_1}\dots \Phi_{k_m}$
of the ring of fractions. Due to \eqref{PP1},\eqref{PP2} this element
does not depend on the choice of the reduced decomposition of $w$.
By \eqref{PX1} to \eqref{PX3},
for all $w\in {\mathcal S}_n$ and $k=1,\ldots,n$
we than have an equality
\begin{equation}
\label{PXk}
\Phi_wX_k=X_{w(k)}\Phi_w.
\end{equation}

The symmetric group ${\mathcal S}_n$ naturally acts on any character $\chi$
of the subalgebra ${\mathcal A}_n\subset\H_n$ so that
$$
(w\cdot\chi)(X_k)=\chi\,(X_{w^{-1}(k)}).
$$
Let $\pi_\chi:\H_n\to\operatorname{End}\,(\h_n)$ be the defining
homomorphism of the $\H_n$-module $I_\chi$.
Note that $\chi(X_{k})\neq0$ for any character $\chi$ and index $k$,
because $X_k^{-1}\in{\mathcal A}_n$.
The character $\chi$ is called {\it regular} if
$\chi(X_{k})\neq\chi(X_{l})$ for $k\neq l$.
For a regular character $\chi$, the action of the algebra
$\H_n$ on $I_\chi$ extends to each element $\Phi_w$
of the ring of fractions.
This extended action is also denoted by~$\pi_{\chi}$.

\begin{proposition}
\label{P1}
For any regular $\chi$,
the operator of
right multiplication in $\h_n$ by the element\/ $\pi_{\chi}(\Phi_w)(1)$
is an intertwining operator $I_{\,w\cdot\chi}\to I_\chi$
of\/ $\H_n\,$-modules.
\end{proposition}

\begin{proof}
Denote by $\mu$ this operator.
The action of the generators $T_{1},\ldots,T_{n-1}$
on the representation space $\h_n$ of $I_\chi$ and $I_{\,w\cdot\chi}$
is through left multiplication and commutes with the operator $\mu$.
It therefore remains for us to verify that the action of the elements
$X_1,\ldots,X_n$ commutes with $\mu$ as well.
Since the vector $1\in\h_n$ is cyclic for the actions
of the subalgebra $\h_n\subset\H_n$ on $I_\chi$ and $I_{\,w\cdot\chi}$,
it is sufficient to demonstrate that the composition
operators $\pi_{\chi}(X_{k})\,\mu$ and $\mu\,\pi_{w\cdot\chi}(X_{k})$
coincide on the identity vector for each $k=1,2,\ldots,n$.
But this follows from \eqref{PXk}:
\begin{align*}
\pi_{\chi}(X_{k})(\pi_{\chi}(\Phi_w)(1))
&=\pi_{\chi}(X_{k}\Phi_w)(1)
\\
&=\pi_{\chi}(\Phi_wX_{w^{-1}(k)})(1)
\\
&=\pi_{\chi}(\Phi_w)(\pi_{w\cdot\chi}(X_{k})(1)).
\qedhere
\end{align*}
\end{proof}

Let us now fix an element $w\in {\mathcal S}_n$ and a point
$(x_1,\ldots,x_n)\in\mathbb C(q)^{\times n}$
such that $x_k\neq x_l$ for $k\neq l$, and $x_k\neq0$ for all $k$.
A regular character $\chi$ can then be determined by setting
$\chi(X_k)=x_{w(k)}$ for $k=1,\ldots,n$.
We have $(w\cdot\chi)(X_k)=x_k$.

\begin{proposition}
\label{P2}
We have the equality\/ $\pi_{\chi}(\Phi_w)(1)=\v_w(x_1,\ldots,x_n)$ in $\h_n$.
\end{proposition}

\begin{proof}
We will use the induction on the length $\ell(w)$ of the element
$w\in\mathcal S_n$. By definition, $\ell(w)$ is the number of
factors in any reduced decomposition of $w$. If $\ell(w)=0$
then $w$ is the identity element of $\mathcal S_n$, and
Proposition~\ref{P2} is trivial.

Now suppose that Proposition~\ref{P2} is valid for some element
$w\in\mathcal S_n$ and every regular character $\chi$.
Take any index $l\in\{1,\ldots,n-1\}$ such that
$\ell(s_lw)>\ell(w)$. Determine a character $\chi'$ of ${\mathcal A}_n$
by setting $\chi'(X_k)=x_{s_lw(k)}$ for $k=1,\ldots,n$.
This character is regular.
We shall make the induction step by showing that
$$
\pi_{\chi'}(\Phi_{s_lw})(1)=\v_{s_lw}(x_1,\ldots,x_n)\,.
$$

Let $(x'_1,\ldots,x'_n)\in\mathbb C(q)^{\times n}$ be the point
obtained from $(x_1,\ldots,x_n)$ by swapping the coordinates
$x_l$ and $x_{l+1}$. Then
$\chi'(X_k)=x'_{w(k)}$ for $k=1,\ldots,n$ so that
$$
\pi_{\chi'}(\Phi_w)(1)=\v_w(x'_1,\ldots,x'_n)=
{}^{s_l}\v_w(x_1,\ldots,x_n)
$$
by the induction assumption.
Further, by the definition of the character $\chi'$ we get 
$$
\chi'(X_{w^{-1}(l)})=x_{l+1}
\quad\text{and}\quad
\chi'(X_{w^{-1}(l+1)})=x_l\,.
$$
Hence by using \eqref{PXk}
\begin{align*}
\pi_{\chi'}(X_l\,\Phi_w)(1)&=x_{l+1}\,\pi_{\chi'}(\Phi_w)(1)\,,
\\
\pi_{\chi'}(X_{l+1}\,\Phi_w)(1)&=x_l\,\pi_{\chi'}(\Phi_w)(1)\,.
\end{align*}
Therefore
\begin{align*}
\pi_{\chi'}(\Phi_{s_lw})(1)
&=\pi_{\chi'}(\Phi_l\,\Phi_w)(1)
\\
&=\v_l(x_1,\ldots,x_n)\,\pi_{\chi'}(\Phi_w)(1)
\\
&=\v_l(x_1,\ldots,x_n)\,{}^{s_l}\v_w(x_1,\ldots,x_n)
\\
&=\v_{s_lw}(x_1,\ldots,x_n)\,.
\qedhere
\end{align*}
\end{proof}


\subsection{Cyclic generators for irreducible $\H_n$-modules}

For any $\l\in\mathcal C_n$
define a character $\chi_\l$ of ${\mathcal A}_n$
by setting 
\begin{equation}
\label{chi}
\chi_\l(X_k)=q^{c_k}\end{equation} 
for each $k=1,\ldots n$.
This character is regular, if and only if
no diagonal of the diagram $\l$ contains more than one box.
However,
using Propositions~\ref{P1}~and~\ref{P2} we obtain the following
corollary to Theorem \ref{T3}.

\begin{corollary}
\label{erole}
If $\l\in\mathcal C_n$
then the operator of right multiplication in $\h_n$ by
the element\/ $E_\l$
is an intertwining operator 
$I_{\chi_\l}\to I_{\,w_0\cdot\chi_\l}$
of\/ $\H_n\,$-modules.
\end{corollary}

\begin{proof}
Take any point $(x_1,\ldots,x_n)\in\mathcal F_\l$
such that $x_k\neq x_l$ for $k\neq l$, and $x_k\neq0$ for all $k$.
A regular character $\chi$ can then be determined by setting
\begin{equation}
\label{gpoint}
(w_0\cdot\chi)(X_k)=x_k
\end{equation}
for each $k=1,\ldots n$.
By choosing $w=w_0$ 
in Propositions \ref{P1} and \ref{P2}, we obtain that 
the operator of right multiplication in $\h_n$ by 
$\v_0(x_1,\ldots,x_n)$
is an intertwining operator $I_{\,w_0\cdot\chi}\to I_{\chi}$.
So is the operator of right multiplication by the product
$$
\delta_\l(x_1,\ldots,x_n)\,\v_0(x_1,\ldots,x_n).
$$
On the other hand, at the point \eqref{point}
the character $w_0\cdot\chi$ defined by the equalities
\eqref{gpoint} specializes to $\chi_\l\,$, see \eqref{chi}.
The character $\chi$ then specializes to $w_0\cdot\chi_\l\,$.
We now get Corollary~\ref{erole} by the definition
\eqref{E} of the element $E_\l\,$.
\end{proof}

Consider the left ideal in $\h_n$ generated by the element $E_\l$.
Corollary~\ref{erole} shows that this left ideal is a submodule of
the induced $\H_n$-module $I_{\,w_0\cdot\chi_\l}$. Let us denote
by $V_\l$ this submodule. Note that $V_\l\neq\{0\}$ because
$E_\l\neq0$. The following theorem has been stated in \cite{C2}
without proof.

\begin{theorem}
\label{main}
The\/ $\H_n$-module $V_\l$ is irreducible.
\end{theorem}

\begin{proof}
Using the notation of Subsection \ref{1.2}, write
$\l=(\l_1,\ldots,\l_r)$ where $\l_i=[a_i,b_i]$ for each
$i=1,\ldots,r$. Then consider another diagram
$\bar\l=(\bar\l_1,\ldots,\bar\l_r)$
with the same number $r$ of rows, such that for each $i=1,\ldots,r$
$$
\bar\l_i=[\,r-b_{r-i+1}+1,r-a_{r-i+1}+1\,]\,.
$$
Then $\bar\l$ has degree $n$, but is not
necessarily a Cherednik diagram.
If $\bar c_1,\ldots,\bar c_n$ are the contents corresponding to
$\bar\l$ then $\bar c_k=-c_{n-k+1}$ for $k=1,\ldots,n$.

We claim that an analogue of Theorem \ref{T3}
holds for the diagram $\bar\l$ instead of $\l$.
Namely, the restriction of the rational function
$\delta_{\,\bar\l}\v_0$ to the
subspace $\mathcal F_{\bar\l}$ is regular and non-zero at the point
\begin{equation}
\label{barpoint}
(x_1,\ldots,x_n)=(q^{\bar c_1},\ldots,q^{\bar c_n})\,.
\end{equation}
Further,
let $\o_n$ be the involutive automorphism of the $\mathbb C(q)$-algebra
$\h_n$ defined by setting $\o_n(T_k)=T_{n-k}$ for each $k=1,\ldots,n-1$.
We also claim that the value 
of that restriction at the point \eqref{barpoint} equals $\o_n(E_\l)$.

To verify these two claims, consider the transformation of
$(\mathbb C(q)\setminus\{0\})^{\times n}$,
$$
(x_1,\ldots,x_n)\mapsto(x_n^{-1},\ldots,x_1^{-1})\,.
$$
This transformation maps the point \eqref{point} to the point
\eqref{barpoint}, and also maps $\mathcal F_\l$ to $\mathcal F_{\bar\l}$.
But by using Lemma \ref{L2}, we get the relation
$$
\v_0(x_n^{-1},\ldots,x_1^{-1})=\o_n(\v_0(x_1,\ldots,x_n)).
$$
Moreover, by the definition of function $\delta_{\bar\l}$ we have
the relation
$$
\delta_{\,\bar\l\,}(x_n^{-1},\ldots,x_1^{-1})=
\delta_\l(x_1,\ldots,x_n).
$$
Our two claims now follow from Theorem \ref{T3} and from the definition
of $E_\l$.

Further, we have an analogue of Corollary \ref{erole} for the
diagram $\bar\l$ instead of $\l$. Consider the character
$\chi_{\bar\l}$ of ${\mathcal A}_n$ corresponding to
$\bar\l$. The operator of right multiplication in $\h_n$ by
the element $\o_n(E_\l)$ is then an
intertwiner $I_{\chi_{\bar\l}}\to I_{\,w_0\cdot\chi_{\bar\l}}$.
Let $V_{\,\bar\l}$ be the left ideal in $\h_n$ generated by the element
$\o_n(E_\l)$. This left ideal is a submodule of
the induced $\H_n$-module $I_{\,w_0\cdot\chi_{\bar\l}}$.
We claim that the irreducibility of the $\H_n$-module
$V_{\,\bar\l}$ is equivalent to that of $V_\l$.

Indeed, the automorphism $\o_n$ of $\h_n$ extends
to an involutive automorphism of the $\mathbb C(q)$-algebra
$\H_n$ by setting $\o_n(X_k)=X_{n-k+1}^{-1}$ for each
index $k=1,\ldots,n$. Consider the $\H_n$-module
$I_{\,w_0\cdot\chi_\l}^{\,\o_n}$ obtained from
$I_{\,w_0\cdot\chi_\l}$ by twisting the latter module with
the automorphism $\o_n$ of $\H_n$. The twisted $\H_n$-module
is equivalent to $I_{\,w_0\cdot\chi_{\bar\l}}$:
the underlying vector space of the two modules is $\h_n$, and
the equivalence map
\begin{equation}
\label{emap}
I_{\,w_0\cdot\chi_\l}^{\,\o_n}
\to
I_{\,w_0\cdot\chi_{\bar\l}}
\end{equation}
can be chosen as $\,\o_n:\h_n\to\h_n$. Here we also use the
equalities for $k=1,\ldots,n$
$$
(w_0\cdot\chi_\l)(\o_n(X_k))
=q^{-c_k}=
(w_0\cdot\chi_{\bar\l})(X_k).
$$
The image of the submodule
$V_\l^{\,\o_n}\subset I_{\,w_0\cdot\chi_\l}^{\,\o_n}$
under the map \eqref{emap} is $V_{\bar\l}$. Therefore
the irreducibility of
$V_{\,\bar\l}$ is equivalent to that of $V_\l^{\,\o_n}$,
and hence to that of $V_\l$.

There is
an involutive automorphism of $\H_n$ as $\,\mathbb C$-algebra,
defined by mapping
$$
q\mapsto q^{-1},\quad
T_k\mapsto-\,q^{-1}\,T_k,\quad
X_k\mapsto X_k
$$
for all possible indices $k$. Denote by $V_\l'$ the $\H_n$-module
obtained by twisting $V_{\bar\l}$ with this automorphism.
We shall establish the irreducibility of $V_\l'$
under the conditions of Definition \ref{CD}
on the diagram $\l$. The irreducibility
of $V_\l$ will then follow. We will use the representation theory of the
quantum enveloping algebra $U_\upsilon(\widehat{\mathfrak{sl}}_N)$
of the Kac-Moody Lie algebra $\widehat{\mathfrak{sl}}_N$,
with the parameter $\upsilon=q^{\,1/2}$.

A link between the representation theories of the affine Hecke algebras
and quantum affine algebras was discovered by Drinfeld \cite{D}.
For the algebras $\H_n$ and $U_q(\widehat{\mathfrak{sl}}_N)$
this link was established by Cherednik \cite{C2,Ch}.
We will employ a version of this link due to Chari and Pressley
\cite{CP}. However unlike in \cite{CP}, here $q$
is a formal parameter, not  a complex number.
Hence we regard $U_\upsilon(\widehat{\mathfrak{sl}}_N)$ as
a $\mathbb C(\upsilon)$-algebra.

There is a functor $\mathcal{J}$
from the category of all finite-dimensional
$\H_n$-modules to the category of finite-dimensional
$U_\upsilon(\widehat{\mathfrak{sl}}_N)$-modules
\cite[Theorem 4.2]{CP}. If $N>n$, then the
$U_\upsilon(\widehat{\mathfrak{sl}}_N)$-module
$\mathcal{J}(V)$ is non-zero for each non-zero $\H_n$-module $V$.
But under the conditions on the diagram $\bar\l$ implied by
Definition \ref{CD},
the $U_\upsilon(\widehat{\mathfrak{sl}}_N)$-module
$\mathcal{J}(V_\l')$ is irreducible 
\cite[Corollary~2.3, Proposition~3.5 and Theorem~4.1]{AK}.
Hence $V_\l$ is also irreducible. Note that the irreducibility of
the $U_\upsilon(\widehat{\mathfrak{sl}}_N)$-module $\mathcal{J}(V_\l')$
can also be derived from \cite[Proposition 3.1]{NT}.
\end{proof}


\section*{2.\ Beginning of the proof of  Theorem \ref{T3}}

\setcounter{section}{2}
\setcounter{subsection}{0}
\setcounter{equation}{0}
\setcounter{theorem}{0}

\subsection{Combinatorial preliminaries}
\label{comb}
Let $u_1,\ldots,u_n$ be the standard basis in the Euclidean vector
space $\mathbb R^n$. Take the root system in $\mathbb R^n$ 
of type $A_{n-1}$. A choice of the set $\Dp$
positive roots is made as
$$
\Dp=\{u_i-u_j\mid\ 1\le i<j\le n\}.
$$
With this choice, the simple roots are $\a_i=u_i-u_{i+1}$
for $i=1,\ldots,n-1$.
We will be interested into certain subsets of  $\Dp$ and 
certain  total orders on them. They have been studied independently 
in \cite{Dyer,Khor,P} where they appear under different names:
total reflection orders, normal orders, compatible orders respectively.

\vbox{
\begin{definition}
\label{convex} 
{\bf (a)} 
A subset $\mathcal{L}\subset \Dp$ is called biconvex  if both 
$\mathcal{L}$ and $\Dp\setminus\mathcal{L}$ are closed under root addition.
\par\noindent
{\bf (b)} 
A total order $<$ on a biconvex set $\mathcal{L}$ 
is said to be a convex order if it satisfies the following conditions:
\begin{enumerate}
\item[(i)] 
if $\a,\b\in\mathcal{L}$, $\a+\b\in \Dp$ and $\a<\b$ , then  $\a<\a+\b<\b$;
\item[(ii)] 
if $\a+\b\in\mathcal{L}$ and $\a\notin\mathcal{L}$, then $\b<\a+\b$.
\end{enumerate}
\end{definition}
}

Note that $\Dp$ is in canonical bjiection with the set
$\{(i,j)\mid 1\le i<j\le n\}$ by $u_i-u_j\mapsto (i,j)$.
We shall tacitly use this identification in the following, 
speaking of {\it pairs} rather than of roots.  We say that 
two pairs $\a,\beta$ are orthogonal
if the corresponding roots are orthogonal in $\mathbb{R}^n$.
Then we write $\a\perp \beta$.

We will be interested into two particular convex orderings.
The lexicographic order on $\Dp$ will be denoted  by $<_1\,$:
\begin{equation}
\label{o1}
(i,j)<_1(k,l)
\ \ \text{if and only if}\ \  
i<k\ \text{or}\ i=k\ \text{and}\ j<l.
\end{equation}
We will also use another useful order which will be denoted by $<_2\,$:
\begin{equation}
\label{o2}
(i,j)<_2(k,l)
\ \ \text{if and only if}\ \   
j<l\ \text{or}\ j=l\ \text{and}\ i<k.
\end{equation}

The symmetric group $\mathcal{S}_n$ acts on the root system,
via permutations of the basis vectors $u_1,\ldots,u_n\in\mathbb{R}^n$.
For any $w\in\mathcal{S}_n$ take the set
$$
\mathcal{I}_{\,w}=\{\a\in \Dp\mid w^{-1}(\a)\notin\Dp\}.
$$
This is the set of  \textit{inversions} for $w^{-1}$.
It is well known that, if
$w=\smm$ is a reduced decomposition, then
$\mathcal{I}_{\,w}=\{\b_1,\ldots,\b_m\}$ where 
$\b_k=s_{i_1}\dots s_{i_{k-1}}(\a_{i_k})$
for $k=1,\ldots m$. Here $m=\ell(w)$. The choice of the
reduced decomposition for $w$ provides a total ordering of $\mathcal{I}_{\,w}$:
here $\b_i<\b_k$ if and only if $i<k$. Furthermore, 
a subset $\mathcal{L}\subset \Dp$ is biconvex if and only if it is of the
form $\mathcal{I}_{\,w}$ for some $w\in \mathcal{S}_n$, see \cite{P}.
The convex orders on $\mathcal{I}_{\,w}$ are exactly the ones provided
by the reduced decompositions of $w$.
The order $<_1$ on $\Dp$
is provided by the decomposition
\begin{equation}
\label{w0_1}
w_0=(s_1\dots s_{n-1})\ldots(s_1\,s_2)\,s_1\,,
\end{equation}
while the order $<_2$ on $\Dp$ is provided by the decomposition
\begin{equation*}
w_0=s_1\,(s_2\,s_1)\ldots (s_{n-1}\ldots s_1)\,.
\end{equation*}

Finally, we recall the following technical result
\cite[Proposition 1.9]{GP}.
Let $<$ be any total order on $\Dp$.
For any $\a\in\Dp$ define the set $\a^\le$ 
as $\{\b\in \Dp\mid \b\le \a\}$. Then 
define the set $\a^\ge$ in the obvious way.

\begin{lemma}
\label{modifiyorder}
Suppose that in a certain convex order\/ $<'$ on\/ $\Dp$  
the pair $(i,j)$ is covered  by $(i+1,j)$ or
covers $(i,j-1)$. Then there exists a convex order\/ $<$ 
on\/ $\Dp$ in which $(i,j)$ covers $(i,i+1)$ or is 
covered by $(j-1,j)$ respectively, 
and which restricts to the order\/ $<'$ on 
$(i,j)^\ge$ or on $(i,j)^\le$ respectively.
\end{lemma}


\subsection{Yang-Baxter relations}
\label{subsection2}
With a slight abuse of  notation, set
\begin{equation}
\label{bsing}
\langle\beta\rangle = \langle\,{x_j}/{x_i}\,\rangle
\end{equation}
for each $\b=(i,j)\in \Dp$,
see Subsection \ref{thefun}.
Then our basic function \eqref{basic} 
can be written as $\v_k=T_k+\langle\a_k\rangle$. If
$w=s_{i_1}\ldots s_{i_m}$ is a reduced decomposition, then
\begin{align*}
\v_w&=(\v_{i_1})
(\phantom{}^{s_{i_1}}\v_{i_2})
(\phantom{}^{s_{i_1}s_{i_2}}\v_{i_3})
\ldots
(\phantom{}^{s_{i_1}\dots s_{i_{m-1}}}\v_{i_m})
\\
&=
(T_{i_1}+\langle\b_1\rangle)\,
(T_{i_2}+\langle\b_2\rangle)\,
(T_{i_3}+\langle\b_3\rangle)\,\dots\,
(T_{i_n}+\langle\b_n\rangle)
\end{align*}
where as above $\b_k=s_{i_1}\ldots  s_{i_{k-1}}(\a_{i_k})$
for each $k=1,\ldots,m$.
Hence, if we denote
$$
\v^l_{\b}=T_l+\langle\b\rangle
$$ 
then we get
\begin{equation}\label{phir}
\v_{w}
=\v^{i_1}_{\b_1}
\v^{i_2}_{\b_2}
\,\ldots\,
\v^{i_m}_{\b_m}
.
\end{equation}
To simplify our notation, 
we will write $\v_{\b_j}$ for $\v^{i_j}_{\b_j}$.  
Also, if $\b=(i,j)\in \Dp$ we will sometimes write $\v_\b=\v_{ij}$. 
A direct consequence of Lemma \ref{L2} is the following

\begin{lemma}
\label{braid}
The function $\v_{w}$ is invariant  under the
following moves of adjacent factors in the product \eqref{phir}:
\begin{align}
\label{perp}
&\v_\a\v_\b=\v_\b\v_\a 
&&
\text{if}\ \ \a\perp \b\,;
\\
\label{YB}
&\v_{\a}\v_{\a+\b}\v_{\b}=\v_{\b}\v_{\a+\b}\v_{\a}
&&
\text{if}\ \ \a+\b\in\Dp.
\end{align}
\end{lemma} 

Another easy observation which will be used later is the following

\begin{lemma}
\label{almostbraid}
If $\a,\b\in\Dp$ are adjacent in a convex order then
$\a\perp\b$ implies that $\v_\a=\v_\a^k\,,\,\v_\b=\v_\b^l$ 
for some $k,l$ with
$|k-l|>1$. If  $\a,\a+\b,\b\in\Dp$  
are adjacent in a convex order
then  $\v_\a=\v_\a^k\,,\,\v_{\a+\b}=\v_{\a+\b}^{k\pm 1}\,,\,\v_\b=\v_\b^k$
for some $k$.
\end{lemma}


\subsection{Singular pairs}

We now begin working towards our proof of Theorem \ref{T3}.
Fix a diagram $\l\in\mathcal C_n$.
For any numbers $i$ and $j$ such that $1\le i<j\le n$, the
pair $(i,j)$ will be called \textit{singular} if $c_i=c_j$.
So the singularity of $(i,j)$ means that $i$ and $j$
occur in the same diagonal of $\L$. If $\b=(i,j)$ is a 
singular pair then the function \eqref{bsing} is singular
at the point \eqref{point}, hence the terminology.\vskip5pt
The next lemma is the core of the fusion procedure. 
It appeared in \cite{N} in the symmetric group setting, and 
can be proved by a direct computation using~\eqref{H3}. 
Note that the middle factors of the products 
appearing at the left hand sides
of \eqref{primaf}, \eqref{secondaf} are singular at 
\eqref{point} but the products, upon restricting to 
$\mathcal F_\l$, are~not.

\begin{lemma}
\label{ffusion} 
Suppose that the pair $(i,j)$ is singular.
\\{\bf(a)} 
If\/ $i,i+1$ belong to the same row of\/ $\L$, 
then for any $k=1,\ldots,n-2$ 
\begin{equation}
\label{primaf}
(\,\v_{i,i+1}^{\,k}\v_{i,j}^{\,k+1}\v_{i+1,j}^{\,k})|_{\,\mathcal F_\l}
(q^{c_1},\ldots,q^{c_n})=
(1+T_k)\,(T_{k+1}T_k-q\,T_{k+1}-q).
\end{equation}
{\bf(b)}
If $j-1,j$ belong to the same row of\/ $\L$, 
then for any $k=1,\ldots,n-2$ 
\begin{equation}
\label{secondaf}
(\,\v_{i,j-1}^{\,k}\v_{i,j}^{\,k+1}\v_{j-1,j}^{\,k})|_{\,\mathcal F_\l}
(q^{c_1},\ldots,q^{c_n})=
(T_kT_{k+1}-q\,T_{k+1}-q)\,(1+T_k).
\end{equation}
\end{lemma}

\medskip

Write $\l=(\l_1,\ldots,\l_r)$ and $\l_k=[a_k,b_k]$ for $k=1,\ldots,r$. 
For $1\le k\le l\le r$ put
\begin{equation}
\label{notazione}
\Dp_{kl}=
\{(i,j)\mid 1\le i<j\le n
\ \textrm{and $i,j$ are in rows $k,l$ of $\L$ respectively} 
\}.
\end{equation}
Then the set $\Dp$ becomes a disjoint union of the subsets 
\eqref{notazione}.
Define a total order $<$ on the set $\Dp$ by
first defining an order on every subset \eqref{notazione}:
the restriction of the order $<$ to
$\Dp_{kl}$ is $<_1$ if $b_k-k=b_l-l$, and is $<_2$ otherwise;
see \eqref{o1} and \eqref{o2}.
Note that if the row $k$ or $l$ of $\l$ has
length one, then the restrictions of
$<_1$ and $<_2$ to $\Dp_{kl}$ are the same.
Now order the subsets \eqref{notazione} relative to each other: 
\begin{equation}
\label{dpp}
\Dp_{11}<\Dp_{12}<\Dp_{22}
<\ldots<
\Dp_{1r}<\Dp_{2r}<\ldots<\Dp_{rr}\,.
\end{equation}
The so defined order $<$ on $\Dp$ will be called \textit{special}.
A straightforward verification of the conditions (i\,,\,ii)
in Definition \ref{convex} shows that the special order $<$ is convex.
Let $w_0=s_{i_1}\dots s_{i_m}$ be the corresponding reduced 
decomposition.
As before, let $\b_1,\ldots,\b_m$ be the elements
of $\Dp$ written in this order.
Here $m=n(n-1)/2$.

\begin{lemma}
\label{lemmafund1}
If $\b_k=(i,j)$ then $i_k=j-i$.
\end{lemma} 

\begin{proof}
Recall that the order $<_1$ on $\Dp$ 
corresponds to the reduced decomposition \eqref{w0_1}.
For this reduced decomposition, the statement of the
lemma is clearly true. On the other hand,
it is easy to see that the special order can be obtained from
$<_1$ by a sequence of switches of adjacent orthogonal pairs.
Switching two orthogonal pairs corresponds to
switching two adjacent simple transpositions from $\mathcal S_n$
in a reduced decomposition of $w_0$. 
Hence the statement is true for the order $<$ too.
\end{proof}

Denote by $d_\l$ the number of singular pairs for $\l$.
Denote by $p_\l$ the number of pairs of
distinct non-empty parallel rows of $\l$.
Any pair of distinct parallel rows of length $l$
gives rise to $l$ singular pairs.
In particular, we have $d_\l\ge p_\l\,$.

Further, denote by $\Rl$ the collection of all singular pairs for $\l$
except the pairs $(i,j)$ where $i$ and $j$ are the first numbers 
in two parallel rows of $\L$.
The set $\Rl$ consists of exactly $d_\l-p_\l$ elements.

Let $\xi=(i,j)=\a_i+\ldots+\a_{j-1}$ be a singular pair for $\l$. 
If\/ $i,i+1$ belong to the same row of $\L$, then
denote $\xi^+=(i+1,j)$ so that $\xi=\a_i+\xi^+$.
Similarly, if $j-1,j$ belong to the same row of $\L$, then
denote $\xi^-=(i,j-1)$ so that $\xi=\a_{j-1}+\xi^-$.

\begin{lemma}
\label{lemmafund2}
For each\/ $\xi\in\Rl$ one can choose\/ $e(\xi)\in\{\pm\}$
such that\/ $\xi$ and\/ $\xi^{\,e(\xi)}$ are adjacent 
in the special order, while for all\/ $\xi\in\Rl$ the elements\/
$\xi^{\,e(\xi)}$ are~distinct.
\end{lemma}

\begin{proof}
For any $\xi\in\Dp$ there is a unique subset 
$\Dp_{kl}\subset\Dp$ containing $\xi\,$;
here $k\le l$. Write $\xi=(i,j)\,$; 
the numbers $i$ and $j$
occur in the rows $k$ and $l$ of $\L$ respectively.
Define the function $\xi\mapsto e(\xi)$
so that $e(\xi)=-$ if $b_k-k=b_l-l$, 
and $e(\xi)=+$ otherwise.
This is a function on the set $\Dp$.
Let us show that the restriction of this
function to the subset $\Rl\subset\Dp$ has all the required properties. 

Let $\xi\in\Rl$. Then $k<l\,$; the numbers $i$ and $j$ 
occur in the same diagonal of $\L$. 

First suppose that 
$b_k-k=b_l-l$. If $a_k-k=a_l-l$ then the rows $k$ and $l$ of $\L$
are parallel. Then the number $j-1$ belongs to the same 
row $l$ of $\L$ as $j$, because $j$ cannot be the first
number in its row. If $a_k-k\neq a_l-l$ then $a_k-k>a_l-l$
by Definition \ref{CD}, and again $j-1$ belongs to the same 
row $l$ as $j$. Hence the pair 
$\xi^-=(i,j-1)\in\Dp_{kl}$ is defined, and the two pairs
$\xi^-,\xi$ are adjacent in the order $<_1\,$.
The latter order coincides with $<$ on the subset 
$\Dp_{kl}\subset\Dp\,$.

Now suppose that $b_k-k\neq b_l-l\,$. Then $b_k-k>b_l-l$ 
by Definition \ref{CD}. Then the number $i+1$ belongs to the same 
row $k$ of $\L$ as $i$, because the numbers $i$ and $j$ occur in
the same diagonal of $\L$. Hence the pair 
$\xi^+=(i+1,j)\in\Dp_{kl}$ is defined, and the two pairs
$\xi,\xi^+$ are adjacent in the order $<_2\,$.
The latter order coincides with $<$ on the subset 
$\Dp_{kl}\subset\Dp\,$. 

Now let $\xi$ run through the set $\Rl$.
The pairs of the form $\xi^-$ are different from each other,
and so are the pairs of the form $\xi^+$.
Moreover, the pairs of the form $\xi^-$
belong to the subsets \eqref{notazione} with 
$b_k-k=b_l-l$, while the pairs of the form $\xi^+$
belong to the subsets \eqref{notazione} with $b_k-k>b_l-l$. 
Since all the subsets \eqref{notazione} are disjoint, all pairs
of the form $\xi^-$ are different from all those of the form $\xi^+$.
\end{proof}

For example,
consider the Cherednik diagram  $\l=([1,2],[2,3],[2,3])\,$;
see \eqref{diaggg}. Here $d_\l=4$ and $p_\l=1$.
The special order on $\Dp$ is
given by the following table:
\begin{align*}
&\Dp_{11}:\quad(1,2),\\
&\Dp_{12}:\quad(1,3),\,
(1,4),\,
(2,3),\,
(2,4),\\
&\Dp_{22}:\quad(3,4),\\
&\Dp_{13}:\quad(1,5),\,(2,5),\,(1,6),\,(2,6),\\
&\Dp_{23}:\quad(3,5),\,(4,5),\,(3,6),\,(4,6),\\
&\Dp_{33}:\quad(5,6).
\end{align*}
Here the collection 
$\Rl$ consists of all singular pairs except the pair
$(1,3)$. In the next table all singular pairs are set in bold; 
the pairs $\xi\in\Rl$
are underlined together with their corresponding pairs $\xi^{\,e(\xi)}\,$:

\begin{align*}
&\Dp_{11}:\quad(1,2),\\
&\Dp_{12}:\quad{\bf(1,3)},\,(1,4),\,\underline{(2,3),\,{\bf(2,4)}},\\
&\Dp_{22}:\quad(3,4),\\
&\Dp_{13}:\quad(1,5),\,(2,5),\,\underline{{\bf(1,6)},\,(2,6)},\\
&\Dp_{23}:\quad(3,5),\,(4,5),\,\underline{{\bf(3,6)},\,(4,6)},\\
&\Dp_{33}:\quad(5,6).
\end{align*}

In the proof of Lemma \ref{lemmafund2},
the value $e(\b)$ was defined for any $\b\in\Dp$.
Hence we can uniquely divide $\Dp$
into new ordered subsets 
$\Dp_1,\ldots,\Dp_s$ so that:
\begin{enumerate}
\item[(i)] 
$\Dp_1<\ldots<\Dp_s\,$;
\item[(ii)]
the function $e(\b)$ is constant on each of the subsets
$\Dp_1,\ldots,\Dp_s\,$;
\item[(iii)]
the subsets $\Dp_1,\ldots,\Dp_s\,$ are maximal with the properties (i\,,\,ii).
\end{enumerate}
Each of the subsets
$\Dp_1,\ldots,\Dp_s$ is a union of certain subsets \eqref{notazione}.
For the diagram $\l=([1,2],[2,3],[2,3])$ from
the previous example, we have $s=3$ and
$$
\Dp_1=\Dp_{11}\sqcup\Dp_{12}\sqcup\Dp_{22}\,,
\quad
\Dp_2=\Dp_{13}\sqcup\Dp_{23}\,,
\quad
\Dp_3=\Dp_{33}\,.
$$
For any Cherednik diagram $\l$ and $t=1,\ldots,s$
the subset $\Dp_1\sqcup\ldots\sqcup\Dp_t$ is biconvex. 

\begin{proposition}
\label{regular} 
Restriction of the function $\dd_\l\v_{0}$ 
to $\mathcal F_\l$ is regular at \eqref{point}.
\end{proposition}

\begin{proof} 
We have already observed that the restriction to $\mathcal F_\l$
of any factor $\v_\beta$ of $\v_{0}$ with a non-singular $\b$
is regular at the point \eqref{point}.
Let us explain the idea of the proof in the particular case
$d_\l=1$. First suppose that $p_\l=1$. Then it suffices to consider
the diagram
$\l=\{[1,1],[2,2]\}$. Here
$$
\dd_\l\v_0\,(x_1,x_2)=
(1-x_2/x_1)\left(T_1+\frac{1-q}{1-x_2/x_1}\right)=
(1-x_2/x_1)\,T_1+{1-q}\,,
$$
which is regular and has the value $1-q$ at the point $(x_1,x_2)=(1,1)$.

Now let $p_\l=0$, then $\dd_\l=1$.
The set $\Rl$ consists of a singular pair $\xi$.
Write $\xi=(i,j)$ and suppose that $e(\xi)=-\,$; 
the case of $e(\xi)=+$ will be similar.
By Lemmas \ref{lemmafund1} and \ref{lemmafund2}, the product $\v_0$
written using the order $<$ on $\Dp$ has the~form
\begin{equation}
\label{underlin}
\ldots\,
\underline{\v_{i,j-1}^{\,k}\,\v_{i,j}^{\,k+1}}
\,\ldots\,
\v_{j-1,j}^1\,\ldots
\end{equation}
where $k=j-i-1$ and we underlined the factors corresponding to 
$\xi^-=(i,j-1)$ 
and $\xi$. Using Lemmas \ref{modifiyorder} and \ref{braid}, 
modify the order of pairs $(i,j),\ldots,(j-1,j)$ 
and the order of the corresponding
factors of the product $\v_0$ to write this product~as
\begin{equation}
\label{threefac}
\ldots\,\v_{i,j-1}^{\,k}\,\v_{i,j}^{\,k+1}\,\v_{j-1,j}^{\,k}\,\ldots
\end{equation}
where we also used Lemma \ref{almostbraid}.
Restrict the three factors in \eqref{threefac} to $\mathcal F_\l$ 
and note that the last one becomes $1+T_k$ upon restriction.
Now use Lemma \ref{ffusion}(b) and get, by 
evaluating at \eqref{point} the restriction to $\mathcal F_\l$ 
of the product of the three factors,
$$
(T_kT_{k+1}-q\,T_{k+1}-q)\,(1+T_k).
$$
We can then take the factor $\v_{j-1,j}$ 
back to its original position  
and evaluate the restrictions to $\mathcal F_\l$
of other elementary factors. In this way we obtain
$$
\v_0\,|_{\,\mathcal F_\l}(q^{c_1},\ldots,q^{c_n})=
A\,(T_kT_{k+1}-q\,T_{k+1}-q)\,B\,,
$$
$$
A=\prod_{(k,l)<(i,j-1)}
\v_{kl}\,|_{\,\mathcal F_\l}(q^{c_1},\ldots,q^{c_n})\,,
\quad
B=\prod_{(k,l)>(i,j)}
\v_{kl}\,|_{\,\mathcal F_\l}(q^{c_1},\ldots,q^{c_n})\,.
$$
So the final effect of our procedure has been to make a ``fusion" of the
two adjacent elementary factors underlined in \eqref{underlin}
into the three-term factor 
\begin{equation}
\label{ttf}
F_{k,k+1}=T_kT_{k+1}-q\,T_{k+1}-q\,,
\end{equation}
and to evaluate
all other elementary factors at \eqref{point} straightforwardly.

Now suppose that the set $\Rl$ consists of a singular pair $\xi=(i,j)$
with $e(\xi)=+\,$; we still assume that $p_\l=0$. 
By Lemmas \ref{lemmafund1} and \ref{lemmafund2}, the product $\v_0$
has the~form
\begin{equation}
\label{underlin+}
\ldots\,
\v_{i,i+1}^1\,
\,\ldots\,
\underline{\v_{ij}^{\,k+1}\,\v_{i+1,j}^{\,k}}
\,\ldots
\end{equation}
where we underlined the factors corresponding to 
$\xi$ and $\xi^+=(i+1,j)$. Here $k=j-i-1$ again.
Modify the order of pairs $(i,i+1),\ldots,(i,j)$ 
and the order of the corresponding
factors of the product $\v_0$ to write this product~as
\begin{equation}
\label{threefac+}
\ldots\,\v_{i,i+1}^{\,k}\,\v_{ij}^{\,k+1}\,\v_{i+1,j}^{\,k}\,\ldots
\end{equation}
where we used Lemma \ref{almostbraid}.
Restrict the three factors in \eqref{threefac+} to $\mathcal F_\l$ 
and note that the first one becomes $1+T_k$ upon restriction.
Now use Lemma \ref{ffusion}(a) and get, by 
evaluating at \eqref{point} the restriction to $\mathcal F_\l$ 
of the product of the three factors,
$$
(1+T_k)\,(T_{k+1}T_k-q\,T_{k+1}-q).
$$
Take the factor $\v_{i,i+1}$ 
back to its original position
and evaluate the restrictions to $\mathcal F_\l$
of other elementary factors. In this way we obtain
$$
\v_0\,|_{\,\mathcal F_\l}(q^{c_1},\ldots,q^{c_n})=
C\,(T_{k+1}T_k-q\,T_{k+1}-q)\,D\,,
$$
$$
C=\prod_{(k,l)<(i,j)}
\v_{kl}\,|_{\,\mathcal F_\l}(q^{c_1},\ldots,q^{c_n})\,,
\quad
D=\prod_{(k,l)>(i+1,j)}
\v_{kl}\,|_{\,\mathcal F_\l}(q^{c_1},\ldots,q^{c_n})\,.
$$
Here the final effect of our procedure has been to make a ``fusion" of
two adjacent elementary factors underlined in \eqref{underlin+}
to the three-term factor 
$$
T_{k+1}T_k-q\,T_{k+1}-q\,,
$$ 
and to evaluate all other elementary factors at \eqref{point} straightly.

To deal with the general case, we first observe that the factor
$\dd_\l$ takes care of the singularities coming from the singular pairs 
which are not in $\Rl$, as in the case $p_\l=1$ above.
Hence we have only to prove that we can perform the procedure 
used in the case $p_\l=0$ in a coherent way for all singular
pairs from the set $\Rl$. By Lemma \ref{lemmafund2} 
the pairs $\xi^{\,e(\xi)}$ are all different when $\xi$ 
ranges over $\Rl$. Hence it suffices to prescribe the order 
on $\Rl$ in which the previous procedure should be performed.
Once this order is given, one can just repeat the 
argument for each pair in $\xi\in\Rl$.

Take any pair $\xi\in\Rl$ and write $\xi=(i,j)$.
There is a unique subset $\Dp_t\subset\Dp$ 
containing $\xi\,$.
Suppose that $e(\xi)=-\,$.
The key combinatorial fact is that then all
$\b\in\Dp$ with $\xi\le\b\le(j-1,j)$ 
belong to the same subset $\Dp_t\subset\Dp$.
Here we use the special order on $\Dp$.
This key fact follows from the first claim of Lemma \ref{comblemma}.

Write $\Dp=\Dp_1\sqcup\ldots\sqcup\Dp_s$ and suppose that
$e(\Dp_s)=\{-\}$\,.  Then perform the previous procedure 
for all singular pairs in $\Dp_s$
starting from the first to the last;
here we refer to our special order on $\Dp_s$. 
Each time we modify the order of the pairs
following the singular pair which we are dealing with,
and then restore the original order of these following pairs.
In particular, each time we do not affect the singular pairs
we dealt with previously.

Here we have $e(\Dp_{s-1})=\{+\}$\,. 
Let us now deal with the singular pairs in $\Dp_{s-1}$
starting from the last to the first 
singular pair in $\Dp_{s-1}\,$, this procedure
takes place in $\Dp_1\sqcup\ldots\sqcup\Dp_{s-1}\,$.
After that we can deal with the singular pairs contained in 
$\Dp_{s-2}$ starting as above from the first to the last
singular pair. By the key fact, the latter
procedure can performed within the subset $\Dp_{s-2}$
so that the pairs from $\Dp_{s-1}\sqcup\Dp_s$
are not affected by this procedure. Now we can proceed inductively.
An obvious modification of the previous argument 
works in the case $e(\Dp_s)=\{+\}$\,.
\end{proof}

\begin{corollary}
\label{limiting}  
The element $E_{\l}\in\h_n$ can be calculated as follows:
\begin{enumerate}
\item[(1)] 
arrange the elementary factors $\v_\b$ in \eqref{phir} 
according to the special order;
\item[(2)] 
replace every two adjacent factors corresponding
to $\b\in\Rl$ by a three-term factor;
\item[(3)]
replace each factor indexed by a singular pair $\b\notin\Rl$  
by the scalar\/ $1-q\,$;
\item[(4)] 
evaluate the factors corresponding to the remaining non-singular pairs.
\end{enumerate}
\end{corollary}

For example, consider again the diagram 
$\l=([1,2],[2,3],[2,3])\,$; see \eqref{diaggg}. Here
\begin{gather*} 
E_\l=
(\dd_\l\v_0)|_{\,\mathcal F_\l}(1,q,1,q,q^{-1},1)=
\\ 
(T_1+1)(1-q)(T_3+1)(T_1T_2-q\,T_2-q)(T_1+1)
(T_4-q)(T_3-q^2(q+1)^{-1})\,\times
\\
(T_5T_4-q\,T_5-q)(T_2-q)(T_1-q^2(q+1)^{-1})(T_3T_2-q\,T_3-q)(T_1+1).
\end{gather*}


\section{End of the proof  of Theorem \ref{T3}}

\setcounter{section}{3}
\setcounter{subsection}{0}
\setcounter{equation}{0}
\setcounter{theorem}{0}

Now we complete the proof of  Theorem \ref{T3} by showing that $E_\l\neq0$.
If the Cherednik diagram $\l$ has no parallel rows, then
Corollary \ref{limiting} immediately shows that the equality
\eqref{lead0} holds for some coefficients $a_w\in\mathbb C(q)$. 
This equality implies that $E_\l\neq0$. However, the equality
\eqref{lead0} does not hold always. For instance, let us 
again consider the diagram $\l=([1,2],[2,3],[2,3])$; see
\eqref{diaggg}. The first two rows of this diagram
are parallel. A direct computation gives
$$
E_\l=q\,(q^2-1)\,
T_1T_3T_4T_3T_5T_4T_2T_1T_3T_2T_1+
\sum_{\ell(w)<11}a_w\,T_w
$$
while $\ell(w_0)=15$. We shall give a similar presentation of 
$E_\l\in\h_n$ for any $\l\in\mathcal{C}_n\,$; see \eqref{desired}. Like \eqref{lead0},
this presentation immediately shows that $E_\l\neq0\,$.

We start with giving another expression for the element
$E_\l$ when the diagram $\l\in\mathcal{C}_n$ consists of one row
only. Then we will employ the notation $E_n$
instead of $E_\l$. As usual, denote 
$[m]_q=(1-q^m)/(1-q)$ for each positive integer~$m$. 

\begin{lemma}
\label{Lemma 1}
We have
\begin{equation}
\label{en}
E_n=\sum_{w\in\mathcal{S}_n}T_w\,,
\end{equation}
\begin{equation}
\label{en2}
E_n^{\,2}=[1]_q\ldots[n]_q\,E_n\,.
\end{equation}
\end{lemma}

\begin{proof} 
By definition,
$$
E_n=\prod_{(i,j)\in\Dp}
\left(T_{j-i}+\frac{1-q}{1-q^{j-i}}\right)
$$
where the pairs $(i,j)\in\Dp$ are ordered lexicographically.
Using Lemma \ref{braid}, for any $k=1,\ldots,n-1$
we can write the product $E_n$ as $(T_k+1)E$ for some
$E\in\h_n$. By \eqref{H3}, we then have
$T_kE_n=q\,E_n$ for every $k=1,\ldots,n-1$.
By \cite[Lemma 2.4]{M}, then $E_n$ equals the right hand side
of \eqref{en} multiplied by a scalar from $\mathbb{C}(q)$.  
This scalar is actually $1$ because the coefficient at $T_{w_0}$
is $1$ for both sides of \eqref{en}. 
In view of \eqref{en}, the equality \eqref{en2}
is well known; see
for instance \cite {M} once again.
\end{proof}

Using the natural embedding $\h_m\to\h_n$,
we will regard $E_m$ as an element of $\h_n$ whenever $1\le m\le n$.
More generally, for any non-negative integer $h$ such that
$h+m\le n$, we have an embedding $\iota_h:\h_m\to\h_n$ such that
$\iota_h(T_k)=T_{k+h}$ for any $k=1,\ldots,m-1$. Denote
$$
E_m^{\,(h)}=\iota_h(E_m).
$$

\begin{lemma}
\label{Lemma 2}
Suppose that the diagram $\l\in\mathcal{C}_n$ consists of\/
$r$ parallel rows of length\/ $m$ each. 
Then
$$
E_\l=f_m^{\thinspace r(r-1)/2}\,E_m^{\,(0)}E_m^{\,(m)}\ldots E_m^{\,(n-m)},
$$
where
\begin{equation*}
\label{fm}
f_m=(-1)^m\,q^{\,m(m-1)/2}\,(q^m-1)\,.
\end{equation*}
\end{lemma}

\begin{proof}
If $r=1$, the lemma is tautological. 
Suppose that $r=2$, so that $n=2m$.
Recall the notation \eqref{ttf}.
By Corollary \ref{limiting}, 
$$
E_\l=E_m\,(1-q)\,ZE_m
$$ 
where $Z$ stands for the product
\begin{align*}
\v_{m+1}(q^0,q^1)\,\v_{m+2}(q^0,q^2)\,\v_{m+3}(q^0,q^3)\,
\v_{m+4}(q^0,q^4)\,\ldots\,\v_{2m-1}(q^0,q^{m-1})
&\,\times
\\
F_{m-1,m}\,
\v_{m+1}(q^1,q^2)\,\v_{m+2}(q^1,q^3)\,
\v_{m+3}(q^1,q^4)\,\ldots\,\v_{2m-2}(q^1,q^{m-1})
&\,\times
\\
\v_{m-2}(q^2,q^0)\,F_{m-1,m}\,\v_{m+1}(q^2,q^3)\,
\v_{m+2}(q^2,q^4)\,\ldots\,\v_{2m-3}(q^2,q^{m-1})
&\,\times
\\
\v_{m-3}(q^3,q^0)\,\v_{m-2}(q^3,q^1)\,F_{m-1,m}\,
\v_{m+1}(q^3,q^4)\,\ldots\,\v_{2m-4}(q^3,q^{m-1})
&\,\times
\\
&\hspace{5pt}\vdots
\\
\v_2(q^{m-2},q^0)\,\ldots\,\v_{m-2}(q^{m-2},q^{m-4})\,
F_{m-1,m}\,\v_{m+1}(q^{m-2},q^{m-1})
&\,\times
\\
\v_1(q^{m-1},q^0)\,\,\ldots\,
\v_{m-3}(q^{m-1},q^{m-4})\,\v_{m-2}(q^{m-1},q^{m-3})\,
F_{m-1,m}
&\,.
\end{align*}

Note that there are exactly $m$ lines in the above display.
If $m=1$ then $E_m=1$ and $Z=1$, so that $E_\l=(1-q)$ as required.

If $m\ge2$, consider the second line in the display.
The element $E_m$ is divisible on the right by
$T_{m-1}+1$, while the product in the first line
of the display commutes with $T_{m-1}+1$. But 
$(T_{m-1}+1)\,T_{m-1}=q\,(T_{m-1}+1)$ and therefore
$$
(T_{m-1}+1)\,F_{m-1,m}=-q\,(T_{m-1}+1)\,. 
$$
Hence in the second line of the display, we can
replace the factor $F_{m-1,m}$ by the scalar factor $-q\,$,
without changing the value of the product $E_m\,Z$.

If $m\ge3$, also consider
the third line in the display.
The element $E_m$ is divisible on the right by
$T_{m-2}+1$, while the product in the first line
of the display commutes with $T_{m-2}+1$.
After replacing $F_{m-1,m}$ by $-q$
in the second line, all factors in that line then commute
 with $T_{m-2}+1$. But
\begin{align*}
(T_{m-2}+1)\,\v_{m-2}(q^2,q^0) 
&=(T_{m-2}+1)\left(T_{m-2}+\frac{1-q}{1-q^{-2}}\right)
\\
&=(T_{m-2}+1)\left(q+\frac{1-q}{1-q^{-2}}\right)
=\frac{q-q^2}{1-q^2}\,(T_{m-2}+1)\,.
\end{align*}
Hence in the third line of the display, we can now
replace the factor $\v_{m-2}(q^2,q^0)$ by the scalar 
$(q-q^2)/(1-q^2)$,
without changing the value of the product $E_m\,Z$.
After this replacement, we apply the arguments like
those already applied to the second line, and replace
$F_{m-1,m}$ in the third line by $-q\,$,
without changing 
$E_m\,Z$.

By continuing these arguments for all $m$ lines of the display, 
we show that 
\begin{equation}
\label{eye}
E_\l=E_m(1-q)\,YE_m
\end{equation}
where $Y$ stands for the product
\begin{align*}
\v_{m+1}(q^0,q^1)\,\v_{m+2}(q^0,q^2)\,\v_{m+3}(q^0,q^3)\,
\v_{m+4}(q^0,q^4)\,\ldots\,\v_{2m-1}(q^0,q^{m-1})
&\,\times
\\
(-q)\,
\v_{m+1}(q^1,q^2)\,\v_{m+2}(q^1,q^3)\,
\v_{m+3}(q^1,q^4)\,\ldots\,\v_{2m-2}(q^1,q^{m-1})
&\,\times
\\
\frac{q-q^2}{1-q^2}\,(-q)\,\v_{m+1}(q^2,q^3)\,
\v_{m+2}(q^2,q^4)\,\ldots\,\v_{2m-3}(q^2,q^{m-1})
&\,\times
\\
\frac{q-q^3}{1-q^3}\,\frac{q-q^2}{1-q^2}\,(-q)\,
\v_{m+1}(q^3,q^4)\,\ldots\,\v_{2m-4}(q^3,q^{m-1})
&\,\times
\\
&\hspace{5pt}\vdots
\\
\frac{q-q^{m-2}}{1-q^{m-2}}\,\ldots\,\frac{q-q^2}{1-q^2}\,
(-q)\,\v_{m+1}(q^{m-2},q^{m-1})
&\,\times
\\
\frac{q-q^{m-1}}{1-q^{m-1}}\,\,\ldots\,
\frac{q-q^3}{1-q^3}\,\frac{q-q^2}{1-q^2}\,
(-q)
&\,.
\end{align*}
By collecting the scalar factors here and by performing cancellations,
$$
Y=(-1)^{m-1}\,q^{\,m(m-1)/2}\,[1]_q^{-1}\ldots[m-1]_q^{-1}E_m^{\,(m)}.
$$
Since the elements $E_m$ and $E_m^{\,(m)}$ of the algebra $\h_{2m}$ commute,
\eqref{eye} implies that
$$
E_\l
=(-1)^{m-1}\,q^{\,m(m-1)/2}\,[1]_q^{-1}\ldots[m-1]_q^{-1}\,
E_m^{\,2}\,(1-q)\,E_m^{\,(m)}
=f_m\,E_m\,E_m^{\,(m)}
$$
as required when $r=2$.
Here we used the relation \eqref{en2} for $m$ instead of $n$.

It remains to observe that for $r>2$,
Lemma \ref{Lemma 2} follows from the case
$r=2$. Consider the expression which
Corollary~\ref{limiting} provides for 
$E_\l$. This is an ordered product of factors corresponding 
to all pairs $(i,j)\in\Dp$.
According to Step~(1) of the corollary
we use the special order on $\Dp$,
so that the ordered set $\Dp$ is divided into the
subsets \eqref{dpp}. Then we also perform Steps (2\,,\,3\,,\,4)
which do not affect division into the subsets. 
The last two subsets in \eqref{dpp} are $\Dp_{r-1,r}$ and $\Dp_{rr}$.
The product of the factors of $E_\l$ corresponding to 
$\Dp_{rr}$ equals $E_m$. Further, consider
the factors
in our expression for $E_\l$ that precede the factors
corresponding to $\Dp_{r-1,r}$.
By considering 
the subset $\Dp_{r-1,r-1}$ and by using Lemma \ref{braid},
one can show that the product of these preceding
factors is divisible by $E_m$ on the right.
By applying to the last two rows of $\l$
the arguments used in the case $r=2$,
we can modify our expression for $E_\l$
without changing the value of that expression. 
Namely, we replace the factors corresponding to
$\Dp_{r-1,r}$ and $\Dp_{rr}$ by $f_m$ and $E_m^{\,(m)}$
respectively. 

Next we replace the factors corresponding to the subset
$\Dp_{r-2,r}$ by $f_m$ and change to $E_m^{\,(2m)}$ 
the element $E_m^{\,(m)}$ previously
appeared on the right of our expression. 
By continuing this process, we replace by
$
f_m^{\thinspace r-1}\,E_m^{\,(n-m)}
$
all factors corresponding to 
$\Dp_{1r}\,\ldots,\Dp_{rr}$ in the expression 
for $E_\l$, initially provided by Corollary \ref{limiting}.   
An induction on the number $r$ 
of parallel rows completes the proof.
\end{proof}

Let us now deal with an arbitrary Cherednik diagram $\l$.
Take any $(i,j)\in\Dp$ and suppose that
the numbers $i$ and $j$
occur in the rows $k$ and $l$ of $\L$ respectively. Here $k\le l\,$.
If $k=l$, denote by $h$ the quantity of numbers in $\L$ occurring
in those rows which precede
the row $k$ and are parallel to it. Denote
\begin{equation}
\label{tij}
t_{ij}=
\begin{cases} 
\,s_{j-i} 
&\text{if the rows $k$ and $l$ of $\l$ are not parallel;}
\\
\,1
&\text{if the rows $k$ and $l$ of $\l$ are parallel but $k<l\,$;}
\\
\,s_{j-i+h}
&\text{if $k=l\,$.}
\end{cases}
\end{equation}
Take the ordered product 
\begin{equation}
\label{reddec}
w_\l=\prod_{(i,j)\in\Dp}t_{ij}
\end{equation}
where we use our special order on the set $\Dp$.
For the diagram 
$\l=([1,2],[2,3],[2,3])$ considered before,
$$
w_\l=s_1s_3s_4s_3s_5s_4s_2s_1s_3s_2s_1\,.
$$

\begin{lemma}
\label{Lemma 3}
For any $\l\in\mathcal{C}_n$
the decomposition \eqref{reddec} in\/ $\mathcal{S}_n$ is reduced.
\end{lemma}

\begin{proof}
Take $(i,j)\in\Dp$ and suppose that
the numbers $i$ and $j$
occur in the rows $k$ and $l$ of $\L$ respectively.
For the purpose of this proof, we will say that the pair
$(i,j)$ is of type I\,,\,II or III
if the rows $k$ and $l$ satisfy the conditions 
in the first, second or third line of \eqref{tij} respectively. 
A direct check of the conditions of Definition~\ref{convex}(a)
shows that the set $\mathcal L$ of all pairs of types I\,,\,III is biconvex.
Hence there are $w,w'\in\mathcal S_n$ such that  
$w_0=w\,w'$ and $\ell(w_0)=\ell(w)+\ell(w')$ while
$\mathcal I_w=\mathcal L$. This implies
that by using only
the relations $s_is_j=s_js_i$ for $|i-j|>1$ and
$s_is_{i+1}s_i=s_{i+1}s_is_{i+1}$ for $1\le i\le n-1$,
we can modify the reduced decomposition
\begin{equation}
\label{reddecspec}
w_0=\prod\limits_{(i,j)\in\mathcal P}s_{j-i}
\end{equation} 
to another one such the first  $\ell(w)$ pairs associated 
with the new decomposition are exactly those of types I and III\,.
Note that in \eqref{reddecspec} we used our special order on $\Dp$.
There is a procedure which changes 
\eqref{reddecspec} to a reduced decomposition 
whose first $\ell(w)$ elements are exactly
the simple transpositions from \eqref{reddec},
taken in the same order. It implies at once that 
\eqref{reddec} is a reduced decomposition of $w_\l=w$.

To show how this procedure works, consider a particular
case when the diagram $\l$ consists of four rows, of which
the second and third make a pair of parallel rows,
while the first and the fourth are not parallel to any other rows.
Our arguments will be easy to extend to an arbitrary diagram $\l$.
In this case the set $\Dp$, endowed with the special order, 
is uniquely divided into three subsets $A<B<C$ where
$B$ consists of pairs of type II while each of the subsets
$A$ and $C$ consist of pairs of types I\,,\,III\,.
Let us write $w_0=w_Aw_Bw_C$ accordingly.
We have to move the factor $w_B$ to the right.
In our case $B=\Dp_{23}$ so that
$$
w_B=\prod_{(i,j)\in\Dp_{23}}s_{j-i}\,.
$$
Further divide $C$ as $\Dp_{33}<D$, here
$\Dp_{33}$ consists
of pairs of type III\,. Accordingly,
$$
w_C=\biggl(\ \prod_{(i,j)\in\Dp_{33}}s_{j-i}\biggr)\,w_D\,.
$$
We have the following equalities of products of
sequences of simple transpositions:
$$
w_0\,=\,\,
w_Aw_B\,\biggl(\ \prod_{(i,j)\in\Dp_{33}}s_{j-i}\biggr)\,w_D=
$$
$$
w_A\,\biggl(\ \prod_{(i,j)\in\Dp_{33}}s_{j-i+h}\biggr)\,w_Bw_D=
$$
\begin{equation}
\label{last}
w_A\,\biggl(\ \prod_{(i,j)\in\Dp_{33}}s_{j-i+h}\biggr)\,w_D
\biggl(\ \prod_{(i,j)\in\Dp_{23}}s_{j-i+m}\biggr)
\end{equation}
where $m=\l_4$ and $h=\l_2$ as required in \eqref{tij}.
For the pairs $(i,j)$ from the subsets $A$ and $D$
we have $t_{ij}=s_{j-i}$. Hence the simple transpositions 
in \eqref{last} multiplied over the subsets $A\,,\Dp_{33}\,,D$ 
make precisely the right hand side of \eqref{reddec}.
\end{proof}

For each $m=1,2,\ldots$ denote by $p_m$ the number of pairs
of distinct parallel rows of $\l$ of length $m$. In particular, if 
$\l\in\mathcal{C}_n$ consists of
$r$ parallel rows of length $m$ each, then
$p_m=r\,(r-1)/2$. We have
$p_\l=p_1+p_2+\ldots$ in general.

\begin{proposition}
\label{enonzero} 
For some coefficients $a_w\in\mathbb C(q)$, we have the equality in $\h_n$
\begin{equation}
\label{desired}
E_{\l}=a\,T_{w_\l}+\sum_{\ell(w)<\ell(w_\l)}a_w\,T_w\,.
\end{equation}
where
$$
a=\prod_{m\ge1}
f_m^{\,p_m}\,.
$$
\end{proposition}

\begin{proof}
Take any maximal subsequence of parallel rows  
of $\l$. In particular,
there is no row before or after these rows, parallel to them.
By Lemma \ref{comblemma}, there is no row of $\l$
occuring in between of these parallel rows.
Suppose that this subsequence consists of more than one row.
Let $m$ the length of any row in this subsequence.

Consider the expression which
Corollary~\ref{limiting} initially provides for 
$E_\l$. By using Lemma \ref{Lemma 2} we can modify
this expression, without affecting 
the value of $E_\l$. Let $k$ and $l\,$ be any two
rows from our sequence of parallel rows. In the case $k<l$
we replace by the scalar factor $f_m$ 
the product of all factors in our expression for $E_\l$
corresponding to the pairs $(i,j)\in\Dp_{kl}\,$.
Note that $t_{ij}=1$ for these pairs, see 
\eqref{tij}.

In the case $k=l$ the factor
corresponding to $(i,j)\in\Dp_{kk}$ 
in our initial expression for $E_\l$
is $\v_{ij}^{\,j-i}(q^{\,c_i},q^{\,c_j})$.
We replace this factor by 
$\v_{ij}^{\,j-i+h}(q^{\,c_i},q^{\,c_j})$ where
$h$ is the quantity of numbers in $\L$ occurring
in the rows which precede
the row $k$ and are parallel to it. 
Here we use the maximality of our
sequence of parallel rows. 
In this case $t_{ij}=s_{j-i+h}\,$, see again the definition \eqref{tij}.
Using the so modified expression for $E_\l$
along with Lemma \ref{Lemma 3}, we get
Proposition \ref{enonzero}.
\end{proof}

Here is a corollary to our proof of Proposition \ref{enonzero};
it refines Corollary \ref{limiting}.

\begin{corollary}
\label{shortening} 
The element $E_{\l}\in\h_n$ can be calculated as follows:
\begin{enumerate}
\item[(1)] 
arrange the elementary factors in \eqref{phir} 
according to the special order;
\item[(2)] 
for every row of\/ $\L$ of length\/ $m=1,2,\ldots$
with\/ $h=1,2,\ldots$ entries  
in the rows preceding and parallel to that row,
replace by\/ $E_m^{\,(h)}$
the product $E_m$ of factors corresponding to all pairs\/ $(i,j)$
where both $i$~and~$j$ are in that row;
\item[(3)] 
for every two distinct parallel rows of $\l$ of length $m=1,2,\ldots$
replace by the scalar factor\/ $f_m$
the product of factors corresponding to all $(i,j)$
where $i$~and~$j$ are respectively
in the first and second of the two parallel rows;
\item[(4)] 
replace the two adjacent factors corresponding
to any remaining singular pair $(i,j)$ by a three-term factor;
\item[(5)] 
evaluate the factors corresponding to all remaining non-singular pairs.
\end{enumerate}
\end{corollary}

For example, consider once again the diagram 
$\l=([1,2],[2,3],[2,3])\,$. Here
\begin{gather*} 
E_\l=
q\,(q^2-1)(T_1+1)(T_3+1)
(T_4-q)(T_3-q^2(q+1)^{-1})\,\times
\\
(T_5T_4-q\,T_5-q)(T_2-q)(T_1-q^2(q+1)^{-1})(T_3T_2-q\,T_3-q)(T_1+1).
\end{gather*}


\section*{Acknowledgments}

We are grateful to Ivan Cherednik 
for illuminating conversations,
and to the anonymous referee for helpful remarks.
We are also grateful to George Lusztig and to Nanhua Xi for drawing our
attention to their works \cite{L3} and \cite{X1,X2}. 
Our work has been supported by the EPSRC
and by the London Mathematical Society.



\end{document}